\documentclass[12pt]{amsart}
\usepackage{amsmath,enumitem}
\usepackage[english,  activeacute]{babel}
\usepackage[latin1]{inputenc}
\usepackage{amssymb}
\usepackage{amsthm}
\usepackage{graphicx}
\usepackage{array}
\usepackage{pdflscape}
\usepackage{a4wide}
\usepackage{xcolor, url}
\usepackage{float}

\theoremstyle{plain}
\newtheorem{theorem}{Theorem}
\newtheorem{corollary}[theorem]{Corollary}
\newtheorem{lemma}[theorem]{Lemma}
\newtheorem{proposition}[theorem]{Proposition}

\theoremstyle{definition}
\newtheorem{definition}[theorem]{Definition}

\newtheorem{convention}[theorem]{Convention}
\newtheorem{notation}[theorem]{Notation}
\newtheorem*{remark}{Remark}

\newcommand{\N}{{\mathbb N}}

\def\callan{\textsc{Callan}}%
\def\ocallan{\overline{\textsc{Callan}}}%
\def\gfree{\textsc{$\Gamma$-free}}%
\def\ogfree{\overline{\textsc{$\Gamma$-free}}}%
\def\lonesum{\textsc{Lonesum}}%
\def\olonesum{\overline{\textsc{Lonesum}}}%
\def\FL{F_{\lambda}}%
\def\oeta{\widehat{\eta}}%
\def\dumont{\textsc{Dumont}}%
\def\piros#1{{\color{red}#1}}%
\def\kek#1{{\color{blue}#1}}%

\title[Lonesum and $\Gamma$-free $0$-$1$ fillings of Ferrers shapes]%
{Lonesum and $\bold{\Gamma}$-free $\bold{0}$-$\bold{1}$ fillings of Ferrers shapes}

\author{Be\'ata B\'enyi}
\address{\noindent Faculty of Water Sciences, National University of Public Service, Budapest, HUNGARY}
\email{beata.benyi@gmail.com}

\author{G\'abor V.\ Nagy}
\address{\noindent Bolyai Institute, University of Szeged, HUNGARY}
\email{ngaba@math.u-szeged.hu}

\date{\today}

\begin{document}
\begin{abstract} 
We show that $\Gamma$-free fillings and lonesum fillings of Ferrers shapes are equinumerous by applying a previously defined bijection on matrices
for this more general case and by constructing a new bijection between Callan sequences and Dumont-like permutations.
As an application, we give a new combinatorial interpretation of Genocchi numbers in terms of Callan sequences.
Further, we recover some of Hetyei's results on alternating acyclic tournaments. Finally, we present an interesting result in the case of certain other special shapes.
\end{abstract}

\maketitle

\section{Introduction}
Given a partition $\lambda=(\lambda_1,\lambda_2,\dots,\lambda_n)$ where $\lambda_1\geq\lambda_2\geq\dots\geq\lambda_n$,
the \emph{Ferrers shape} $\FL$ is an arrangement of cells justified to the left and to the bottom such that
$\FL$ has $n$ rows with $\lambda_i$ cells in the $i^{\text{th}}$ row, from bottom to top ($i=1,\dots,n$). See Figure~\ref{fig_FS}.\par
\begin{figure}[h]
\centering
\includegraphics[scale=1.1]{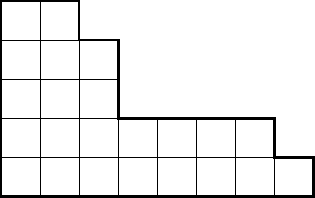}
\caption{The Ferrers shape associated to the partition (8,7,3,3,2)}\label{fig_FS}
\end{figure}
A \emph{Ferrers diagram} (tableau, or $0$-$1$-filling of a Ferrers shape) is an assignment of a $0$ or a $1$ to each of the cells
of a Ferrers shape $\FL$. We call a Ferrers diagram \emph{$\Gamma$-free} if it does not contain $1$'s in positions such that
they form a $\Gamma$-configuration, i.e., two $1$'s in the same row and a third below the left of these in the same column.
Figure~\ref{fig_FD} shows a Ferrers diagram (on the left) which is $\Gamma$-free, and an other one
(on the right) which is not $\Gamma$-free.\par
\begin{figure}[h]
\centering
\includegraphics[scale=1.1]{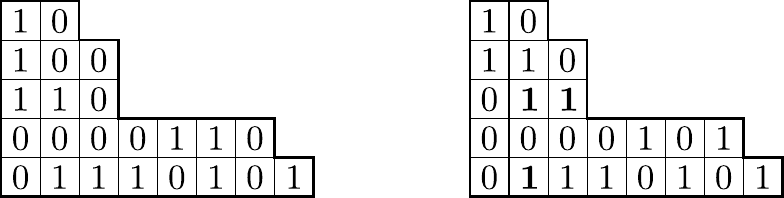}
\caption{A $\Gamma$-free and a non-$\Gamma$-free Ferrers diagram}\label{fig_FD}
\end{figure}
Clearly, we can say that a diagram is $\Gamma$-free if and only if it does not contain any of the submatrices
from the following set:
\begin{align*}
\Gamma=
\left\{\begin{pmatrix}1&1\\1&0\end{pmatrix},
\begin{pmatrix}1&1\\1&1\end{pmatrix}
\right\}.
\end{align*}
In this paper we think of $n\times k$ $0$-$1$ matrices as Ferrers diagrams of rectangular shape, namely, Ferrers diagrams corresponding to 
the partition $(k,k,\dots,k)$ where the number of $k$'s is $n$. So the definitions and theorems concerning Ferrers diagrams apply to $0$-$1$ matrices as well.
$\Gamma$-free diagrams are generalizations of the so-called  $\Gamma$-free matrices that were defined in \cite{FH} and enumerated in \cite{BH1},
showing that the number of $n\times k$ $\Gamma$-free \mbox{$0$-$1$} matrices is the poly-Bernoulli number $B_n^{(-k)}$.
Poly-Bernoulli numbers were introduced by Kaneko \cite{Kaneko} analytically as a generalization of the Bernoulli numbers.
The poly-Bernoulli numbers enumerate several combinatorial objectclasses, as for instance lonesum matrices.
Lonesum matrices are $0$-$1$ matrices uniquely reconstructible by their row and column sum vectors \cite{Brewbaker}.
For a more detailed list of combinatorial objects enumerated by the poly-Bernoulli numbers, see \cite{BH1, BHcombrel}.
In \cite{BNagy} the authors defined a bijection between $\Gamma$-free matrices and the so-called Callan permutations
that are in a simple one-to-one correspondence with lonesum matrices.

In this paper we apply a version of this bijection on $\Gamma$-free diagrams, establishing a bijection between $\Gamma$-free diagrams
and the so-called \emph{lonesum fillings} of diagrams. With that we show that the sets of $\Gamma$-free fillings and lonesum fillings of the same Ferrers shape are equinumerous. 

We call a $0$-$1$ filling of a Ferrers shape \emph{lonesum} if it doesn't contain any of the following two submatrices:

\[\mathcal{F}=\left\{
 \begin{pmatrix} 1 & 0\\0 & 1 \end{pmatrix},
\begin{pmatrix}0 & 1 \\ 1 & 0\end{pmatrix}\right\}.
\]

See Figure~\ref{fig_LS} for an example of lonesum Ferrers diagram. We call the above pair of submatrices also the \emph{flipping pair}. The name is motivated by the fact that the column and row sums of a $0$-$1$ matrix does not change when we exchange any occurance of one of the submatrix of the pair to the other. Moreover, Ryser \cite{Ryser} showed that if two matrices $ A$ and $A^*$ have the same row and column sum vector then $A$ can be transformed to $ A^*$ by a sequence of exchanges between sumbmatrices of the flipping patterns. Similar statements are true for the case of fillings of Ferrers shapes, hence we can state that a lonesum filling of a given shape is uniquely reconstructible from the row and column sum vector. 

Lonesum fillings of Ferrers shapes were already investigated in different forms in the literature. Josuat-Verg\`{e}s \cite{Josuat} refers to them as $X$-diagrams and establishes a bijection between lonesum fillings and Le-tableaux \cite{Postnikov}. On the other hand, there is an obvious bijection between lonesum fillings of a Ferrers shape $F_{\lambda}$ and acyclic orientations of the so-called Ferrers graph, the bipartite graph associated to a given shape $F_{\lambda}$ \cite{Ehrenborg}. The special case of matrices reduces to a one-to-one correspondence between acyclic orientations of complete bipartite graphs and lonesum matrices \cite{Cameron}.
Motivated by the results in \cite{Ehrenborg}, Selig et al.\ defined and studied in \cite{EW} the so-called EW-tableaux and NEW-tableaux that are essentially the same as lonesum fillings of Ferrers shapes.

We study properties of a bijection between $\Gamma$-free and lonesum fillings of Ferrers shapes.  
Further, we consider the special case when the Ferrers shape is a staircase shape. A \emph{staircase shape} is the Ferrers shape associated to the partition $\lambda=(n,n-1,\ldots, 1)$. It follows from the properties of a well-known bijection between Le-tableaux with at least one $1$ in every column and permutations \cite{Steingrimsson} that such Le-tableaux of staircase shape are enumerated by the Genocchi numbers. Based on our bijection we present a new combinatorial interpretation of Genocchi numbers in terms of so-called Callan sequences and prove this result by describing a direct bijection with Dumont permutations that are well-known to be enumerated by the Genocchi numbers. Additionally, using a connection to lonesum fillings of staircase shapes we present a new proof of Hetyei's results on the number of acyclic alternating tournaments, showing that this is given by the Genocchi numbers \cite{Hetyei}.

In the final section we start to study the $0$-$1$ fillings of other types of shapes by showing that permuting the columns of a Ferrers shape does not change
the number of $\Gamma$-free fillings.
  

\section{A bijection between $\Gamma$-free and lonesum Ferrers diagrams}

The main result of this paper is a bijective proof of the fact that the number of $\Gamma$-free $0$-$1$-fillings of a given Ferrers shape
is the same as the number of lonesum $0$-$1$-fillings (see Theorem~\ref{fotetel}). This is done by connecting two former bijective results 
on this subject. The authors of this paper developed a method in~\cite{BNagy} for transforming $\Gamma$-free matrices into Callan sequences,
which quickly yields a bijective encoding of $\Gamma$-free diagrams into special (Callan) sequences; and an other recent paper~\cite{EW} gives
a bijective encoding of lonesum diagrams into special (Dumont) permutations. As a new ingredient, we establish a bijection between
these two intermediate sets of objects, i.e.\ the set of special sequences and the set of special permutations in question (see Lemma~\ref{keylemma}). 

\begin{definition}
We call a row or column of a $0$-$1$ matrix or Ferrers diagram \emph{non-zero} if it contains at least one $1$ element.
We say that a $0$-$1$ matrix or Ferrers diagram is \emph{complete}, if all of its \emph{columns} are non-zero.
\end{definition}

\begin{notation}
Given a Ferrers shape $\FL$, we denote by $\gfree(\FL)$ the set of $\Gamma$-free Ferrers diagrams of shape $\FL$, 
and we denote by $\lonesum(\FL)$ the set of lonesum Ferrers diagrams of shape $\FL$, defined in the Introduction.
We denote by $\ogfree(\FL)$ and $\olonesum(\FL)$, respectively, the set of \emph{complete} $\Gamma$-free and lonesum Ferrers diagrams of shape $\FL$.
When $\FL$ is an $n\times k$ matrix, we write ``$n\times k$'' in place of $\FL$ in these notations,
so for example, $\gfree(n\times k)$ denotes the set of $\Gamma$-free $n\times k$ $0$-$1$ matrices, and so on.
\end{notation}

The main result of this paper is a bijective proof of the following theorem.

\begin{theorem}\label{fotetel}
$|\gfree(\FL)|=|\lonesum(\FL)|$, for any given Ferrers shape $\FL$.
\end{theorem}

We will prove the following variant bijectively.
\begin{theorem}\label{fotetelv}
$\left|\ogfree(\FL)\right|=\left|\olonesum(\FL)\right|$, for any given Ferrers shape $\FL$.
\end{theorem}

It is easy to see that Theorem~\ref{fotetelv} implies Theorem~\ref{fotetel}. This follows from the facts
\begin{flalign*}
&&\gfree(\FL)&=\mathop{\dot{\bigcup_{\mathcal{I}}}}\gfree_{\mathcal{I}}(\FL)\sim\mathop{\dot{\bigcup_{\mathcal{I}}}}\ogfree(\FL|_{\mathcal{I}}),&\\
\text{and}&&\lonesum(\FL)&=\mathop{\dot{\bigcup_{\mathcal{I}}}}\lonesum_{\mathcal{I}}(\FL)\sim\mathop{\dot{\bigcup_{\mathcal{I}}}}\olonesum(\FL|_{\mathcal{I}}),&
\end{flalign*}
where $\dot{\cup}$ denotes disjoint union,
$\mathcal{I}$ runs over all subsets of columns of $\FL$,
furthermore $\gfree_{\mathcal{I}}(\FL)$ and $\lonesum_{\mathcal{I}}(\FL)$ denote the set of
those $\Gamma$-free/lonesum diagrams of shape $\FL$ in which the set of non-zero columns is precisely $\mathcal{I}$.
Finally, $\FL|_{\mathcal{I}}$ is the Ferrers shape obtained from $\FL$ by deleting all columns not in $\mathcal{I}$,
and the symbol $\sim$ indicates the fact that $\gfree_{\mathcal{I}}(\FL)$ and $\lonesum_{\mathcal{I}}(\FL)$ are in natural bijection with
$\ogfree(\FL|_{\mathcal{I}})$ and $\olonesum(\FL|_{\mathcal{I}})$, respectively, since the diagrams in, say, $\gfree_{\mathcal{I}}(\FL)$
are just complete $\Gamma$-free $0$-$1$ fillings of $\FL|_{\mathcal{I}}$,
and the remaining columns of $\FL$ are filled with $0$'s.

Thus in the rest of this section, we will prove Theorem~\ref{fotetelv} by means of a series of lemmas.
The main tool of the proof has been developed in~\cite{BNagy}. We recall the required definitions and an unusual convention (and extend them to Ferrers shapes) first,
and introduce some new ones.

\begin{definition}
An \emph{$(n,k)$-Callan sequence} is a sequence $(R_1,C_1),\dots,(R_m,C_m)$ for some $m\in\N_0$ such that $R_1,\dots,R_m$ are pairwise disjoint nonempty subsets of $\{1,\dots,n\}$,
and $C_1,\dots,C_m$ are pairwise disjoint nonempty subsets of $\{1,\dots,k\}$. The set of $(n,k)$-Callan sequences is denoted by $\callan(n,k)$.
\end{definition}

\begin{convention}\label{konv}
Throughout this paper, the rows of a matrix or Ferrers shape are always indexed {\it from bottom to top\/} and the columns are indexed {\it from right to left\/}.
\end{convention}

\begin{definition}
In a $0$-$1$ matrix or Ferrers diagram, we call an element 1 \emph{top-1} if it is the highest $1$ in its column.
For an $n\times k$ $0$-$1$ matrix $M$, the set of positions of top $1$'s is denoted by $\text{top}(M)$, i.e.
$$\mathrm{top}(M):=\{(i,j)\in\{1,\dots,n\}\times\{1,\dots,k\}: \text {$M_{ij}$ is a top-1 in $M$}\},$$
where $M_{ij}$ is the element in the $i^{\text{th}}$ row and $j^{\text{th}}$ column of $M$, with respect to Convention~\ref{konv}.
\end{definition}

Now we can state the aforementioned main lemma.
\begin{lemma}\label{folemma}\cite{BNagy}
There exists a bijection $\phi$ from $\gfree(n\times k)$ to $\callan(n,k)$,
such that for any matrix $M\in\gfree(n\times k)$, if $\phi$ maps $M$ to the Callan sequence $(R_1,C_1),\dots,(R_m,C_m)$, then, following Convention~\ref{konv},
\begin{enumerate}
\item the set of indices of non-zero rows of $M$ is $\cup_{i=1}^m R_i$;
\item the set of indices of non-zero columns of $M$ is $\cup_{i=1}^m C_i$; and
\item 
$$\mathrm{top}(M)=\bigcup_{i=1}^m\left(\{\max R_i\}\times C_i\right),$$
where $\max R_i$ denotes the largest element in $R_i$.
\end{enumerate}
\end{lemma}
\begin{proof}
For the construction of $\phi$, see the subsection `Proof of Theorem 2' in~\cite{BNagy}.
The required conditions (1)-(3) are listed  in the subsection `Some properties of $\phi$' of that paper as properties (i)-(ii).
\end{proof}

We shall use this lemma to obtain an encoding of Ferrers diagrams in Callan sequences. 
For a given partition $\lambda=(\lambda_1,\dots,\lambda_n)$, let $\mathrm{Rect}(\FL)$ denote 
the rectangular shape (matrix) of size $n\times\lambda_1$. In fact, $\mathrm{Rect}(\FL)$ is the
circumscribed rectangular shape of $\FL$, that can be obtained from $\FL$ by adding new cells to the end of its rows,
as shown in Figure~\ref{fig_RECT} where the new cells are colored grey.\par
\begin{figure}[h]
\centering
\includegraphics[scale=1.1]{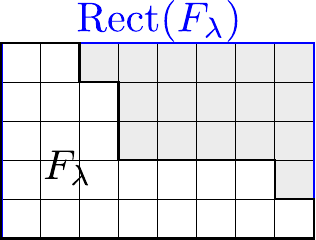}
\caption {$\mathrm{Rect}(\FL)$ and the inscribed $\FL$}\label{fig_RECT}
\end{figure}
It is fairly obvious that in order to obtain the number of $\Gamma$-free Ferrers diagrams of shape $\FL$,
it is equivalent to count those $\Gamma$-free $0$-$1$ fillings of $\mathrm{Rect}(\FL)$ in which all the $1$'s of the matrix
lie inside the inscribed Ferrers shape $\FL$. It is also evident that all the $1$'s lie inside the inscribed $\FL$
if and only if all the top-$1$'s of $\mathrm{Rect}(\FL)$ lie inside $\FL$.
By condition~(3) of Lemma~\ref{folemma}, the bijection $\phi$ can be easily restricted to the set of to-be-enumerated $\Gamma$-free matrices,
and we obtain the following Corollary.
\begin{corollary}\label{kov1}
Given a Ferrers shape $\FL$ with $n$ rows and $k$ columns, the followings are true:
\smallbreak
\noindent\enspace\emph{(a)}\enspace There exists a bijection $\psi$ from $\gfree(\FL)$ to the set of those $(n,k)$-Callan sequences
$(R_1,C_1),\dots,(R_m,C_m)$
for which all elements of the set
$$\bigcup_{i=1}^m\left(\{\max R_i\}\times C_i\right),$$
when a pair $(i,j)$ is interpreted as the position in the $i^\text{th}$ row and $j^\text{th}$ column of $\mathrm{Rect}(\FL)$ (cf.\ Convention~\ref{konv}),
lie inside the inscribed Ferrers shape $\FL$.\smallbreak
\noindent\enspace\emph{(b)}\enspace Moreover, if $D\in\gfree(\FL)$ is mapped to the Callan sequence $(R'_1,C'_1),\dots,(R'_m,C'_m)$ by $\psi$, then
the set of indices of non-zero rows of $D$ is $\cup_{i=1}^m R'_i$, and the set of indices of non-zero columns of $D$ is $\cup_{i=1}^m C'_i$.
\end{corollary}
\begin{proof}
\enspace (a)\enspace See the preceding discussion. (Observe that the size of $\mathrm{Rect}(\FL)$ is $n\times k$.)\smallbreak
\noindent\enspace (b)\enspace Properties (1)-(2) of Lemma~\ref{folemma} also hold for $\psi$, since $\psi$ is just a restriction of $\phi$
to the set of $\Gamma$-free fillings of $\mathrm{Rect}(\FL)$ with $0$'s outside the inscribed $\FL$.
\end{proof}
In this form, Corollary~\ref{kov1} is not very useful to us, even if we rewrite the condition on Callan sequences into a more algebraic form.
We introduce a new labeling of rows and columns of Ferrers shapes instead, which is analogous to the labeling appearing in \cite{EW}.
Label the rows and columns on the northeast border of a Ferrers shape $\FL$ with the numbers $1,2,3,\dots$ such that the bottom row gets label $1$ and
the successive border edges get the remaining numbers in order, as shown in Figure~\ref{fig_CL}. We call this labeling the \emph{canonical labeling} of rows and columns of $\FL$.\par
\begin{figure}[h]
\centering
\includegraphics[scale=1.3]{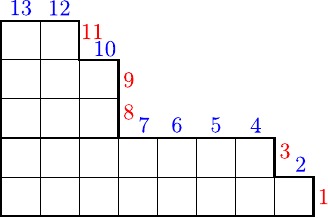}
\caption {Canonical labeling of rows and columns}\label{fig_CL}
\end{figure}
The largest assigned canonical (column) label is called the \emph{semiperimeter} of $\FL$.
The canonical labels of the $i^\text{th}$ row and $j^\text{th}$ column of $\FL$ are denoted by $l_r(i)$ and $l_c(j)$, respectively,
where the indices $i$ and $j$ are interpreted with respect to Convention~\ref{konv}. The sets of all canonical labels of rows and columns of $\FL$ are denoted by $L_r(\FL)$ and $L_c(\FL)$, respectively.
We note that $L_r(\FL)\mathop{\dot{\cup}}L_c(\FL)=\{1,2,\dots,s\}$, where $s$ is the semiperimeter of $\FL$.
\par
We obtain a more tractable form of Corollary~\ref{kov1} if we replace the row/column indices with canonical labels.  It turns out that the diagrams in $\gfree(\FL)$
can be encoded by the following type of sequences.
\begin{definition}\label{def_callan}
For a given Ferrers shape $\FL$, on \emph {$\FL$-Callan sequence} we mean a sequence $(R_1,C_1),\dots,(R_m,C_m)$ for some $m\in\N_0$ such that
\begin{itemize}
\item $R_1,\dots,R_m$ are pairwise disjoint non\-empty subsets of $L_r(\FL)$,
\item $C_1,\dots,C_m$ are pairwise disjoint nonempty subsets of $L_c(\FL)$, and 
\item $\max R_i < \min C_i$, for all $i=1,\dots,m$.
\end{itemize}
The set of $\FL$-Callan sequences is denoted by $\callan(\FL)$.
We say that an $\FL$-Callan sequence $(R_1,C_1),\dots,(R_m,C_m)$ is \emph{complete}, if 
$\cup_{i=1}^m C_i = L_c(\FL)$, i.e.\ if $\{C_1,C_2,\dots,C_m\}$ is a \emph{partition} of $L_c(\FL)$.
For example, if $\FL$ is the Ferrers shape from Figure~\ref{fig_CL}, then the sequence
$$(\{\piros{3}\},\{\kek{4},\kek{7},\kek{10}\}),\,(\{\piros{1}\},\{\kek{2},\kek{5},\kek{6},\kek{12}\}),\,(\{\piros{8},\piros{11}\},\{\kek{13}\})$$
is a complete $\FL$-Callan sequence.
The set of complete $\FL$-Callan sequences is denoted by $\ocallan(\FL)$.
\end{definition}
We have now the following lemma.
\begin{lemma}\label{kov2}
\enspace\emph{(a)}\enspace There exists a bijection $\eta$ from $\gfree(\FL)$ to $\callan(\FL)$, for any given Ferrers shape $\FL$.\smallbreak
\noindent\enspace\emph{(b)}\enspace Moreover, if $D\in\gfree(\FL)$ is mapped to the $\FL$-Callan sequence $(R'_1,C'_1),\dots,(R'_m,C'_m)$ by $\eta$, then
the set of canonical labels of non-zero rows of $D$ is $\cup_{i=1}^m R'_i$, and the set of canonical labels of non-zero columns of $D$ is $\cup_{i=1}^m C'_i$.
\end{lemma}
\begin{proof} This lemma is just Corollary~\ref{kov1} rewritten using canonical labels.
In more detail, assume that $\FL$ has $n$ rows and $k$ columns.
Then $\psi$~establishes a bijection between $\gfree(\FL)$ and the set of $(n,k)$-Callan sequences satisfying the condition of Corollary~\ref{kov1}.a.
We transfrom the $(n,k)$-Callan sequences into $\FL$-Callan sequences, by applying the canonical labeling on the ground sets $\{1,\dots,n\}$ and $\{1,\dots,k\}$:
The $(n,k)$-Callan sequence $(R_1,C_1),\dots,(R_m,C_m)$ is injectively transformed into the sequence 
\[(l_r(R_1),l_c(C_1)),\dots,(l_r(R_m),l_c(C_m)),\] using the standard notation $l(S)=\{l(s):s\in S\}$ if $l$ is a function and $S$ is a set.
(Recall that $l_r(i)$ or $l_c(i)$ gives the canonical label of the $i^\text{th}$ row or column of $\FL$, respectively.)
We further denote this transformation by $\pi$.
Trivially, $\pi(\callan(n,k))$ contains precisely those $(R_1,C_1),\dots,(R_m,C_m)$ sequences which satisfy the first two 
conditions in the definition of $\FL$-Callan sequences. The key observation is that a sequence $\bold{s}\in\callan(n,k)$
satisfies the condition of Corollary~\ref{kov1}.a if and only if $\pi(\bold{s})$ satisfies the third condition in the definition of $\FL$-Callan sequences.
These altogether mean that $\pi$ establishes a bijection between $\psi(\gfree(\FL))$ and $\callan(\FL)$, hence,
$\eta:=\pi\circ\psi$ is a bijection from $\gfree(\FL)$ to $\callan(\FL)$, which proves (a).
Part (b) is an easy consequence of the above. We have that $\eta(D)=\pi(\psi(D))$. Corollary~\ref{kov1}.b shows how to read off the \emph{indices} of non-zero rows and columns of $D$ from $\psi(D)$.
Applying $\pi$ on $\psi(D)$ just replaces every index with the corresponding canonical label. 
\end{proof}

Lemma~\ref{kov2}.a and Lemma~\ref{kov2}.b directly imply the following.

\begin{lemma}\label{Glemma}
There exists a bijection $\oeta$ from $\ogfree(\FL)$ to $\ocallan(\FL)$, for any given Ferrers shape $\FL$.
\end{lemma}
\begin{proof}
The restriction of $\eta$ to complete diagrams is a suitable bijection $\oeta$.
\end{proof}

Now we turn our attention to the other type of $0$-$1$ fillings appearing in Theorem~\ref{fotetelv}, the complete lonesum Ferrers diagrams. We will heavily rely on the encoding
developed in~\cite{EW}.

\begin{definition}\label{ascdesc}
An $n$-permutation is a permutation of the integers $\{1, 2,\dots,n\}$. Fix an $n$-permutation $\alpha=a_1a_2\dots a_n$.
We say that the element $a_i$ is an \emph{ascent bottom} of $\alpha$, if $i\le n-1$ and $a_i<a_{i+1}$.
We say that the element $a_i$ is a \emph{descent top} of $\alpha$, if $i\le n-1$ and $a_i>a_{i+1}$.
The set of ascent bottom elements and the set of descent top elements of $\alpha$ are denoted by $A(\alpha)$ and $D(\alpha)$, respectively.
Note that $A(\alpha)\mathop{\dot{\cup}}D(\alpha)=\{1,2,\dots,n\}\setminus\{a_n\}$.
\end{definition}

\begin{definition}\label{def_dumont}
Given a Ferrers shape $\FL$ with semiperimeter $s$, 
an \emph{$\FL$-Dumont permutation} is an $(s+1$)-permutation $\delta=d_1d_2\dots d_{s+1}$, such that
$d_{s+1}=s+1$ and $A(\delta)=L_r(\FL)$. We note that these conditions also imply that $D(\delta)=L_c(\FL)$.
For example, if $\FL$ is the Ferrers shape from Figure~\ref{fig_CL}, then
$$\kek{5},\kek{4},\kek{2},\piros{1},\piros{9},\kek{13},\piros{8},\kek{12},\kek{10},\kek{7},\kek{6},\piros{3},\piros{11},14$$
is an $\FL$-Dumont permutation.
The set of $\FL$-Dumont permutations is denoted by $\dumont(\FL)$.
\end{definition}
The following statement was proved in~\cite{EW} in a slightly different setting (for example, the Ferrers diagram is reflected, and a different terminology is used).
\begin{lemma}\label{EWlemma}\cite{Steingrimsson}
There exists a bijection $\zeta$ from $\olonesum(\FL)$ to $\dumont(\FL)$, for any given Ferrers shape $\FL$.
\end{lemma}
\begin{proof}
A construction can be extracted from Sections~2-3 of~\cite{EW}. In fact, the bijection we need is not explicitly defined in~\cite{EW}, only very similar
variants, so we define a suitable $\zeta$ for clarity. However, after reading Sections~2-3 of~\cite{EW}, it should be clear to the reader that
our $\zeta$ is well defined, and it is indeed a $\olonesum(\FL)\to\dumont(\FL)$ bijection, as no new ideas are involved at all.\par
So pick an arbitrary diagram $D\in\olonesum(\FL)$ with semiperimeter $s$. 
Repeat the following two steps in the given order on $D$ until every canonical label $\{1,\dots,s\}$ is added to the temporary permutation $\alpha$, which is initially empty:
\begin{enumerate}
\item Add the canonical labels of all (possibly empty) $0$-free rows of the actual $D$ in decreasing order to the end of $\alpha$, and delete the (remaining) elements of these rows from $D$,
\item Add the canonical labels of all (possibly empty) $1$-free columns of the actual $D$ in increasing order to the end of $\alpha$, and delete the (remaining) elements of these columns from $D$,
\end{enumerate}
Of course, for example, on $0$-free row we mean a row without $0$ elements. After the label of a row/column is added to $\alpha$, that row/column is never considered again. It can be seen that 
this procedure continues until every row and column is processed, and we end up with an $s$-permutation $\alpha=a_1a_2\dots a_s$. Then $\zeta(D)$ is defined to be $a_sa_{s-1}\dots a_1(s+1)$.\par
For example, if $D$ is the complete lonesum diagram of Figure \ref{fig_LS}, then $\alpha$ is
$$\piros{11},\piros{8},\kek{12},\piros{9},\kek{10},\kek{13},\piros{3},\piros{1},\kek{2},\kek{4},\kek{5},\kek{6},\kek{7}$$
and so $\zeta(D)$ is
$$\kek{7},\kek{6},\kek{5},\kek{4},\kek{2},\piros{1},\piros{3},\kek{13},\kek{10},\piros{9},\kek{12},\piros{8},\piros{11},14.$$
\end{proof}
\begin{figure}[h]
\centering
\includegraphics[scale=1.3]{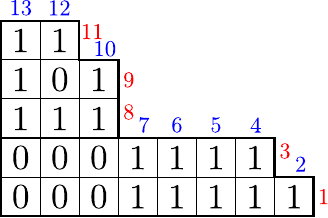}
\caption{A complete lonesum filling}\label{fig_LS}
\end{figure}

In view of Lemma~\ref{Glemma} and Lemma~\ref{EWlemma}, it is enough to construct a $\dumont(\FL)\to\ocallan(\FL)$ bijection in order to prove Theorem~\ref{fotetelv}. This is done in the next lemma, completing the
proof of our main theorem.
\begin{lemma}\label{keylemma}
There exists a bijection $\nu$ from $\dumont(\FL)$ to $\ocallan(\FL)$, for any given Ferrers shape $\FL$.
\end{lemma}
\begin{remark}
We emphasize that this is an elementary lemma involving certain type of sequences and permutations and, in fact,
it has nothing to do with Ferrers shapes or canonical labelings.
In this lemma we can think of $L_r(\FL)$ and $L_c(\FL)$ as two \emph{arbitrary} given sets (denoted by some cryptic notation) such that
$L_r(\FL)\mathop{\dot{\cup}}L_c(\FL)=\{1,2,\dots,s\}$ for some integer $s$, when reading the definitions of $\dumont(\FL)$ and $\ocallan(\FL)$.
\end{remark}
\begin{proof}[Proof of Lemma~\ref{keylemma}]
In the spirit of the Remark, set $S_1:=L_r(\FL)$ and $S_2:=L_c(\FL)$, and assume that
$S_1\mathop{\dot{\cup}}S_2=\{1,\dots,s\}$. We can forget about the Ferrers shape $\FL$.\par
First, we give the definition of a suitable mapping $\nu$.
Pick an arbitrary $(s+1)$-permutation $\alpha=(a_1a_2\dots a_{s+1})\in\dumont(\FL)$. Recall that $a_{s+1}=s+1$
and, for $i\le n$, the element $a_i$ is an ascent bottom if and only if $a_i\in S_1$,
and $a_i$ is a descent top if and only if $a_i\in S_2$, by definition. We will refer to this as the \emph{Dumont property}. We say that an element is of \emph{type j} (where $j=1,2$)
if it is contained in the set $S_j$, and in addition, we define the element $a_{s+1}=s+1$ to be of type $1$.
These two types of elements divide $\alpha$ into blocks, i.e.\ a \emph{block} of $\alpha$ is a maximal sequence of consecutive elements of the same type in $\alpha$.
For example, if $s=20$ and
\begin{align*}
S_1&=\{1,3,5,...,19\},\\
S_2&=\{2,4,6,\dots,20\},\\
\alpha&=(\kek{12},\kek{10},\piros{3},\piros{5},\piros{9},\piros{11},\kek{16},\kek{14},\kek{8},\kek{6},\kek{4},\kek{2},\piros{1},\piros{7},\piros{13},\piros{17},\kek{20},\kek{18},\piros{15},\piros{19},\piros{21}),
\end{align*}
then the blocks of $\alpha$ are $(12,10)$, $(3,5,9,11)$, $(16,14,8,6,4,2)$, $(1,7,13,17)$, $(20,18)$ and $(15,19,21)$.
On the type of a block we mean the common type of its elements. It is obvious that the types of blocks alternate, and $a_{s+1}=s+1$ implies
that the last block is always of type $1$. By the Dumont property, the elements of a block of type $1$ are in increasing order,
and the elements of a block of type $2$ are in decreasing order in $\alpha$.\par
Now we define the transformation of $\alpha$. Initially, the output sequence $\bold{a}$ is empty.
\begin{enumerate}
\item Consider the first block of $\alpha$, and denote this block by $B_1$.
Let $B_2$ denote the next block in $\alpha$ (if no such block exists, we are done).
\item There are two cases:
\begin{enumerate}
\item If $B_1$ is of type $1$ (and $B_2$ is of type $2$), then let $U$ be the set of those elements of $B_2$
that are greater than $\max B_1$, the last element of $B_1$. By the Dumont property, $U$ is not empty (the first element of $B_2$ belongs to $U$),
and the elements of $U$ form an initial segment of $B_2$ (this segment can be the whole $B_2$).
In this case we add the pair $(B_1,U)$ to the end of $\bold{a}$, and remove the elements of $B_1$ and $U$ from $\alpha$ (these elements form an initial segment).
\item If $B_1$ is of type $2$ (and $B_2$ is of type $1$), then let $U$ be the set of those elements of $B_2$
that are smaller than $\min B_1$, the last element of $B_1$. By the Dumont property, $U$ is again not empty and the elements of $U$ form
an initial segment of $B_2$. In this case we add the pair $(U,B_1)$ to the end of $\bold{a}$, and remove the elements of $B_1$ and $U$ from $\alpha$
(these elements form an initial segment).
\end{enumerate}
\item Continue with step~(1) on the truncated $\alpha$, and repeat this procedure until the number of blocks in $\alpha$ decreases to $1$:
If the actual $\alpha$ has only one block in step~(1), then set $\nu(\alpha):=\bold{a}$, and the process terminates.
\end{enumerate}
 Let us demonstrate this procedure on the example permutation $\alpha$ defined in the earlier part of this proof. Here and henceforth, the superscript $(i)$ of $\alpha$ and $\bold{a}$ indicates the actual state after the $i^{\text{th}}$ stage.
\begin{align*}
\alpha&=(\kek{12},\kek{10},\piros{3},\piros{5},\piros{9},\piros{11},\kek{16},\kek{14},\kek{8},\kek{6},\kek{4},\kek{2},\piros{1},\piros{7},\piros{13},\piros{17},\kek{20},\kek{18},\piros{15},\piros{19},\piros{21}).
\intertext {For this $\alpha$, in the first stage case~(2.b) applies with $B_1=\{\kek{12},\kek{10}\}$ and $U=\{\piros{3},\piros{5},\piros{9}\}$, and $(U,B_1)$ is added to the output. We
are left with}
\alpha^{(1)}&=(\piros{11},\kek{16},\kek{14},\kek{8},\kek{6},\kek{4},\kek{2},\piros{1},\piros{7},\piros{13},\piros{17},\kek{20},\kek{18},\piros{15},\piros{19},\piros{21})\\
\quad\bold{a}^{(1)}&:\ (\{\piros{3},\piros{5},\piros{9}\},\{\kek{12},\kek{10}\}).
\intertext {
In the next stage case~(2.a) applies with $B_1=\{\piros{11}\}$ and $U=\{\kek{16},\kek{14}\}$ resulting}
\alpha^{(2)}&=(\kek{8},\kek{6},\kek{4},\kek{2},\piros{1},\piros{7},\piros{13},\piros{17},\kek{20},\kek{18},\piros{15},\piros{19},\piros{21})\\
\quad\bold{a}^{(2)}&:\ (\{\piros{3},\piros{5},\piros{9}\},\{\kek{12},\kek{10}\}),\,(\{\piros{11}\},\{\kek{16},\kek{14}\}).
\intertext{After the next stage we have}
\alpha^{(3)}&=(\piros{7},\piros{13},\piros{17},\kek{20},\kek{18},\piros{15},\piros{19},\piros{21})\\
\quad\bold{a}^{(3)}&:\ (\{\piros{3},\piros{5},\piros{9}\},\{\kek{12},\kek{10}\}),\,(\{\piros{11}\},\{\kek{16},\kek{14}\}),\,(\{\piros{1}\},\{\kek{8},\kek{6},\kek{4},\kek{2}\});
\intertext{and then}
\alpha^{(4)}&=(\piros{15},\piros{19},\piros{21})\\
\quad\bold{a}^{(4)}&:\ (\{\piros{3},\piros{5},\piros{9}\},\{\kek{12},\kek{10}\}),\,(\{\piros{11}\},\{\kek{16},\kek{14}\}),\,(\{\piros{1}\},\{\kek{8},\kek{6},\kek{4},\kek{2}\}),(\{\piros{7},\piros{13},\piros{17}\},\{\kek{20},\kek{18}\}).
\end{align*}
Finally, we conclude that $\nu(\alpha)$ is defined to be the sequence 
$$(\{\piros{3},\piros{5},\piros{9}\},\{\kek{12},\kek{10}\}),\,(\{\piros{11}\},\{\kek{16},\kek{14}\}),\,(\{\piros{1}\},\{\kek{8},\kek{6},\kek{4},\kek{2}\}),(\{\piros{7},\piros{13},\piros{17}\},\{\kek{20},\kek{18}\}).$$
Back to the general discussion, it is clear that this procedure terminates, since the number of blocks of $\alpha$ decreases at each stage.
It is also easy to check that the last element of $\alpha$ (that is, $s+1$) ``survives'' the procedure, so a nonempty final segment of the last block of $\alpha$ remains
at the end. By inspecting the cases (2.a) and (2.b), we can see that whenever the pair $(R,C)$ is added to $\bold{a}$, then
$\emptyset\subsetneq R\subseteq S_1$, $\emptyset\subsetneq C\subseteq S_2$, and $\max R<\min C$. This, together with the ``move some parts of the input to the output'' fashion of the algorithm, yields 
that $\nu(\alpha)$ is an $\FL$-Callan sequence. The completeness of $\nu(\alpha)$ follows from the fact that every block of type $2$ is processed by step~(2),
as the surviving elements of $\alpha$ are of type $1$.\par
Now we sketch why $\nu$ is a bijection. To this end, pick an arbitrary element $\bold{b}\in\ocallan(\FL)$, where
$\bold{b}$ is, say, the sequence $(R_1,C_1),\dots,(R_m,C_m)$. Our goal is to find a unique inverse image $\beta$. So assume that $\nu(\beta)=\bold{b}$ for some
$\beta\in\dumont(\FL)$.
It turns out that $\beta$ can be uniquely reconstructed from $\bold{b}$ by processing the stages of the encoding algorithm in reverse direction, from the last stage to the first stage.
We know that the number of stages is $m$, the number of elements of $\bold{b}$.
We know that the surviving elements of $\beta$ after the last stage are precisely the elements in the set
$$\{1,\dots,s+1\}\setminus\cup_{i=1}^m(R_i\cup C_i),$$
which is a subset of $S_1\cup\{s+1\}$ by the completeness of $\bold{b}$, and these elements form $\beta^{(m)}$ in increasing order.
We know that $\beta^{(m)}$ was obtained from $\beta^{(m-1)}$ by removing an initial segment $\gamma$ of $\beta^{(m-1)}$, i.e.\ $\beta^{(m-1)}=\gamma\beta^{(m)}$,
and this removed initial segment $\gamma$ was ``transformed'' into $(R_m,C_m)$ in step~(2) of the last stage of the algorithm.
By inspection, if step (2.a) was applied, then $\gamma=R_m^{\nearrow}C_m^{\searrow}$ with $B_1=R_m$ and $U=C_m$,
else if step (2.b) was applied, then $\gamma=C_m^{\searrow}R_m^{\nearrow}$ with $B_1=C_m$ and $U=R_m$,
where $R_m^{\nearrow}$ denotes the sequence of elements of $R_m$ in increasing order,
and $C_m^{\searrow}$ denotes the sequence of elements of $C_m$ in decreasing order.
The point is that \emph{exactly one} of these scenarious could happen, as we will see. 
Recall that
\begin{enumerate}[label=(\roman*)]
\item $\max R_m<\min C_m$,
\end{enumerate}
and by the Dumont property of $\beta$,
\begin{enumerate}[resume, label=(\roman*)]
\item the elements of $R_m$ are ascent bottoms in $\beta^{(m-1)}=\gamma\beta^{(m)}$,
\item and the elements of $C_m$ are descent tops in $\beta^{(m-1)}=\gamma\beta^{(m)}$.
\end{enumerate}
If the first element of $\beta^{(m)}$ is greater than $\min{C_m}$, then the first option $\gamma=R_m^{\nearrow}C_m^{\searrow}$
is not possible, because $\beta^{(m-1)}=R_m^{\nearrow}C_m^{\searrow}\beta^{(m)}$ would violate (iii),
as the last element of $C_m^{\searrow}$ would not be a descent top.
Furthermore, $\beta^{(m-1)}:=C_m^{\searrow}R_m^{\nearrow}\beta^{(m)}$, the second option,
satisfies properties (ii)-(iii), and using property (i) and the definition of $U$ in step~(2.b), it is straightforward
to check that in this case step~(2.b) \emph{indeed} transforms 
this $\beta^{(m-1)}$ into $\beta^{(m)}$ and creates the pair $(R_m,C_m)$.\par
Otherwise, if the first element of $\beta^{(m)}$ is less than $\min{C_m}$, then the second option $\gamma=C_m^{\searrow}R_m^{\nearrow}$
is not possible, because then $\beta^{(m-1)}=C_m^{\searrow}R_m^{\nearrow}\beta^{(m)}$ would either violate~(ii) [if $\max R_m^{\nearrow}>\min\beta^{(m)}$],
or $\beta^{(m-1)}$ would not be transformed into $\beta^{(m)}$ [if $\max R_m^{\nearrow}<\min\beta^{(m)}$] because
step~(2.b) would add the first element of $\beta^{(m)}$ to $U$.
Furthermore, $\beta^{(m-1)}:=R_m^{\nearrow}C_m^{\searrow}\beta^{(m)}$, the first option, 
satisfies properties (ii)-(iii), and using property (i) and the definition of $U$ in step~(2.a), it is straightforward
to check that in this case step~(2.a) \emph{indeed} transforms this $\beta^{(m-1)}$ into $\beta^{(m)}$ and creates the pair $(R_m,C_m)$.\par
In sum, we determined the unique $\beta^{(m-1)}$ in both cases. We can repeat this process, $\beta^{(i-1)}$ can be reconstructed uniquely
from $\beta^{(i)}$ and $(R_i,C_i)$, for $i=m,m-1,\dots$, until we reach to a unique permutation~$\beta$.
There are two different cases. When $\beta^{(i)}$ starts with a block of type~$1$, the same argument applies as above.
We leave the reader to verify the analogous case when $\beta^{(i)}$ starts with a block of type~$2$. 
After doing that, it should be clear that the obtained $\beta$ is indeed an $\FL$-Dumont permutation, for which $\nu(\beta)=\bold{b}$,
as required.
\end{proof}

\section{An application: Staircase diagrams and Genocchi numbers}

This section is devoted to the case of staircase shape, a special type of Ferrers shape. In this section $S_n$ denotes the staircase (Ferrers) shape
associated to the partition $(n,n-1,\dots,1)$, see Figure~\ref{fig_SS}.\par
\begin{figure}[h]
\centering
\includegraphics[scale=1.1]{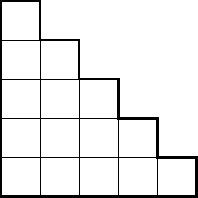}
\caption{The staircase shape $S_5$}\label{fig_SS}
\end{figure}
We know from Theorem~\ref{fotetel} and Theorem~\ref{fotetelv}
that the number of $\Gamma$-free $0$-$1$ fillings of $S_n$ is equal to the number of lonesum $0$-$1$ fillings of $S_n$, and the same holds
if we consider \emph{complete} $0$-$1$ fillings only. It turns out, as it follows from results in \cite{EW} and  \cite{Steingrimsson}, that in the case 
of complete $0$-$1$ fillings a famous combinatorial quantity appears, the Genocchi number. Now we discuss how this can be deduced 
from our results developed in the previous section.
We start with a well-known combinatorial definition of Genocchi numbers \cite{Dumont}.
\begin{definition}\label{def_genocchi}
The unsigned \emph{Genocchi number} $|G_{2n+2}|$ counts the number of those $(2n+1)$-permutations $\alpha$ for which
\begin{itemize}
\item the last element of $\alpha$ is $2n+1$,
\item $A(\alpha)=\{1,3,5,\dots,2n-1\}$, and
\item $D(\alpha)=\{2,4,6,\dots,2n\}$,
\end{itemize}
where $A(\alpha)$ and $D(\alpha)$ are the ascent bottom and descent top sets defined in Definition~\ref{ascdesc}.
We note that permutations satisfying the above conditions are called (classical) \emph{Dumont permutations of the first kind}
in the literature.
\end{definition}
Using this combinatorial definition of Genocchi numbers, we can easily count the number of \emph{complete} lonesum/$\Gamma$-free $0$-$1$ fillings of $S_n$.
(The total number of lonesum/$\Gamma$-free $0$-$1$ fillings of $S_n$ is determined in the next section, see Theorem~\ref{th:Hetyei2} and Corollary~\ref{cor:Hetyei2}.)
\begin{theorem}\label{th_genocchi}
(a)\enspace The number of complete $\Gamma$-free $0$-$1$ fillings of $S_n$ is the Genocchi number $|G_{2n+2}|$.\par
\noindent\enspace(b)\enspace The number of complete lonesum $0$-$1$ fillings of $S_n$ is the Genocchi number $|G_{2n+2}|$.
\end{theorem}
\begin{proof}
By Theorem~\ref{fotetelv}, the two cardinalities are the same, so it is enough to prove the second statement.
Lemma~\ref{EWlemma} establishes a bijection between the set of complete lonesum $0$-$1$ fillings of $S_n$
and the set of $S_n$-Dumont permutations (cf.\ Definition~\ref{def_dumont}). But $S_n$-Dumont permutations are precisely the Dumont permutations of the first kind
in Definition~\ref{def_genocchi}, since $L_r(S_n)=\{1,3,\dots,2n-1\}$ and $L_c(S_n)=\{2,4,\dots,2n\}$, and the semiperimeter of $S_n$ is $2n$.
Hence, the theorem follows.
\end{proof}
Part (a) of the previous theorem and Lemma~\ref{Glemma} (applied on $S_n$) show new interpretations of Genocchi numbers.
\begin{theorem}
The Genocchi number $|G_{2n+2}|$ counts the number of sequences
$$R_1,C_1,R_2,C_2,\dots,R_m,C_m,R_{m+1}$$
for which
\begin{itemize}
\item $\{R_1,\dots,R_{m+1}\}$ is a partition of $\{1,3,5,\dots,2n+1\}$,
\item $\{C_1,\dots,C_{m}\}$ is a partition of $\{2,4,6,\dots,2n\}$,
\item $\max R_i < \min C_i$, for all $i=1,\dots,m$, and
\item $m$ is an arbitrary nonnegative integer.
\end{itemize}
(We note that the third condition implies that $2n+1\in R_{m+1}$.)
\end{theorem}
\begin{proof} 
Observe first that the sequences in the theorem are essentially the same as complete $S_n$-Callan sequences (cf.\ Definition~\ref{def_callan}).
We just transformed the complete $S_n$-Callan sequences
$$(R_1,C_1),(R_2,C_2),\dots,(R_m,C_m)$$
into sequences
$$R_1,C_1,R_2,C_2,\dots,R_m,C_m,R_{m+1}$$
by dropping some parentheses and adding the set $R_{m+1}:=\{1,3,5,\dots,2n+1\}\setminus(\cup_{i=1}^m R_i)$. We note that $R_{m+1}$ is non-empty,
as $2n+1\in R_{m+1}$.\par
After these preliminaries it is enough to refer to Lemma~\ref{keylemma}, which gives a direct
bijection from $\dumont(S_n)$ to $\ocallan(S_n)$, where $\dumont(S_n)$ is the set of Dumont permutations of the first kind, the classical interpretation of $|G_{2n+2}|$,
as it was observed in the proof of Theorem~\ref{th_genocchi}, and $\ocallan(S_n)$ is the set of sequences of this theorem (in a slightly different form, as discussed above).
We note that the worked-out example in the proof of Lemma~\ref{keylemma} illustrates the encoding of a Dumont permutation of the first kind
into a complete $S_n$-Callan sequence.
\end{proof}

\section{Alternation acyclic tournaments}

In this section we connect our results to the recently presented combinatorial interpretations of the Genocchi numbers by Hetyei \cite{Hetyei} The discovery of this new geometrical interpretation was due to Hetyei accidental during his investigations of a homogenous variant of the Linial arrangement. He used analogously to Postnikov and Stanley \cite{PostStan} a class of tournaments for bijectively labeling the regions of the so-called \emph{homogenized Linial arrangements}. Here we present new combinatorial, bijective, proof for a theorem of this work and establish some new connections.

First, we recall the necessary definitions and notation from \cite{Hetyei}. 

A \emph{tournament} on the set $\{1,2,\ldots, n\}$ is a directed simple graph (without loops and multiple edges), such that for each pair of vertices $\{i,j\}$ from $\{1,2,\ldots, n\}$, exactly one of the directed edges $i\rightarrow j$ or $i\leftarrow j$ belongs to the graph. We consider $\{1,2,\ldots, n\}$ with the natural order. A directed edge $i\rightarrow j$ is called an \emph{ascent} if $i<j$ (and is denoted by $i\xrightarrow{a}j$), otherwise it is called a \emph{descent} (denoted by $i\xrightarrow{d} j$).     
A directed cycle $C=(v_0,v_1,\ldots v_{2k-1})$ is \emph{alternating} if ascents and descents alternate along the cycle, i.e., $v_{2j}\xrightarrow{d}v_{2j+1}$ and $v_{2j+1}\xrightarrow{a} v_{2j+2}$ hold for all $j$, the indices taken modulo $2k$. A tournament is \emph{alternation acyclic} if it does not contain any alternating cycle. An important characterization of alternation acyclic tournaments is the following (see Corollary 2.3 in \cite{Hetyei}).
\begin{proposition}\cite{Hetyei}\label{prop:Hetyei}
A tournament $T$ on $\{1,2,\ldots, n\}$ is alternation acyclic if and only if it does not contain any alternating cycle of length 4. 
\end{proposition}
An alternation acyclic tournament $T$ on $\{1,2,\ldots, n\}$ is called \emph{ascending} if every $i<n$ is the tail of an ascent, i.e., for each $i<n$ there is a $j>i$ such that $i\rightarrow j$. 
We present now a new proof for the following theorem.
\begin{theorem}\cite{Hetyei}\label{th:Hetyei}
The number of ascending alternation acyclic tournaments on $\{1,2,\ldots, n\}$ is the unsigned Genocchi number $|G_{2n}|$.
\end{theorem}  

Our proof is based on the observation that the natural coding of an alternation acyclic tournament coincides with a lonesum filling of the staircase shape. 

We associate to a given tournament $T$ on $\{1,2,\ldots, n\}$ a $0$-$1$ filling of a staircase shape $S_{n-1}$ in the following natural way. Label the rows of the staircase shape $S_{n-1}$ by $\{2,3,\ldots, n \}$ from top to bottom, and the columns by $\{1,2,\ldots, n-1\}$ from left to right.
Each cell $(i,j)$ corresponds to an edge of the complete graph $K_n$ on $\{1,2,\ldots, n\}$. Fill the cell $(i,j)$ with $0$ whenever the edge is directed from $i$ to $j$ ($i\rightarrow j$), and with $1$ if the edge is directed otherwise ($i\leftarrow j$). Note that in this coding for every ascent we have a $1$ entry, and for every descent a $0$ in the filling of the staircase shape. Further, a pattern $
\left(\begin{smallmatrix} 1 & 0\\0 & 1 \end{smallmatrix}\right)$, as well as 
$\left(\begin{smallmatrix}0 & 1 \\ 1 & 0\end{smallmatrix}\right)$, corresponds to an alternating cycle of length four. Hence, by Proposition \ref{prop:Hetyei}, we have that the acyclic tournaments on $\{1,2,\ldots,n\}$ are in one-to-one correspondence with the lonesum fillings of the staircase shape $S_{n-1}$ by the above natural coding. Furthermore, the ascending property means in the terms of our coding that each column contains at least one $1$ entry. Hence, ascending acyclic tournaments are in bijection with complete lonesum fillings of staircase shape. We formulate these statements precisely in the next theorem. 
\begin{lemma}\label{th_lonesum-tournaments} Lonesum fillings of staircase shape $S_{n-1}$ are in bijection with alternation acyclic tournaments on $\{1,2,\ldots, n\}$ and complete lonesum fillings of staircase shape $S_{n-1}$ are in bijection with ascending alternation acyclic tournaments on $\{1,2,\ldots, n\}$. 
\end{lemma}
\begin{proof}[Proof of Theorem \ref{th:Hetyei}]
Theorem \ref{th_genocchi}.b and Lemma \ref{th_lonesum-tournaments} imply the Theorem \ref{th:Hetyei}. 
\end{proof}
On the other hand Hetyei's result \cite{Hetyei} implies the following statement, and a new combinatorial interpretation of the median Genocchi numbers. 
\begin{theorem}\label{th:Hetyei2}
The number of lonesum fillings of the staircase shape $S_n$ is the median Genocchi number $H_{2n+1}$. 
\end{theorem}
\begin{proof}
On one hand, using the natural coding above we have a bijection between lonesum fillings of the staircase shape and the alternation acyclic tournaments, on the other hand according to the theorem in \cite{Hetyei} the number of alternation acyclic tournaments on the set $\{1,2,\ldots, n+1\}$ is the median Genocchi number $H_{2n+1}$. Hence, the theorem follows. 
\end{proof}
Theorem~\ref{fotetel} directly implies the following result on the number of $\Gamma$-free fillings of $S_n$.
\begin{corollary}\label{cor:Hetyei2}
The number of $\Gamma$-free fillings of the staircase shape $S_n$ is the median Genocchi number $H_{2n+1}$. 
\end{corollary}
Furthermore, we have an interesting result on the set of lonesum fillings that contains in all row and in all column at least one $1$. 
\begin{theorem}
The number of lonesum fillings of the staircase shape $S_n$ such that each column and each row contains at least one $1$ is the median Genocchi number $H_{2n-1}$.
\end{theorem}
\begin{proof}
The well-known bijection of Steingr\'imsson and Williams \cite{Steingrimsson} between permutations and Le-tableaux with at least one $1$ in each column has the property that all-zero rows correspond to fixed points of the permutation, hence, Le-tableaux of staircase shapes with at least one $1$ in each row and column are in one-to-one correspondence with Dumont derangements, Dumont permutations without fixed points. It is known that these permutations are enumerated by the median Genocchi numbers \cite{Dumont}. On the other hand, Josuat-Verg\`es  \cite{Josuat} established a bijection between lonesum fillings and Le-tableaux of the same shape  that  sends zero-rows (respectively columns) to all-zero rows (respectively columns), so the statement follows.
\end{proof}
\begin{corollary} 
The number of alternation acyclic tournaments on $\{1,2,\ldots, n\}$ such that each $i<n$ is a tail of an ascent and each $i>1$ is the endpoint of an ascent, is the median Genocchi number $H_{2n-1}$. 
\end{corollary}

\section{$0$-$1$ fillings of other types of shapes}
In this section we have started to investigate the fillings of other types of shapes.
The first natural question is whether the number of lonesum/$\Gamma$-free fillings varies if we replace the Ferrers shape $\FL$
with its horizontally or vertically  reflected copies
-- denoted by $\FL^{\leftrightarrow}$ and $\FL^{\updownarrow}$, respectively --, cf.\ Figure~\ref{fig_tukr}).\par
\begin{figure}[h]
\centering
\includegraphics[scale=1.1]{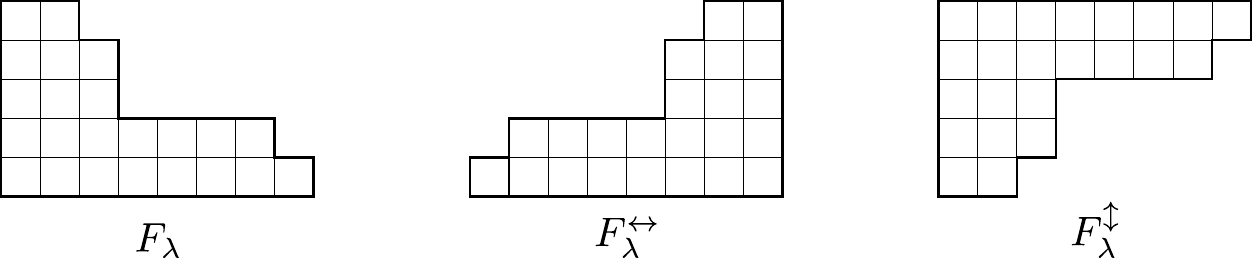}
\caption {A Ferrers shape and its reflected copies}\label{fig_tukr}
\end{figure}
By symmetry, it is immediate that $\FL^{\leftrightarrow}$ and $\FL^{\updownarrow}$ have the same number of
lonesum $0$-$1$ fillings as $\FL$ has, because 
both reflections establish a bijection between the lonesum \emph{diagrams} of shape $\FL$ and the lonesum diagrams of 
the reflected shape.\par
The above question is not so obvious for $\Gamma$-free fillings. For example, the staircase shape $S_2$ shows that
the number of $\Gamma$-free fillings of $\FL^{\updownarrow}$ can differ from $\FL$'s. However, it turns out
that $\FL^{\leftrightarrow}$ has the same number of $\Gamma$-free fillings as $\FL$ has. This follows from the following more general result.
\begin{theorem}
Let $\widetilde{\FL}$ denote an arbitrary shape obtained from the Ferrers shape $\FL$ by permuting its columns (see Figure~\ref{fig_perm} for an example).
Then $\widetilde{\FL}$ and $\FL$ have the same number of $\Gamma$-free $0$-$1$ fillings.
\end{theorem}
\begin{figure}[h]
\centering
\includegraphics[scale=1.1]{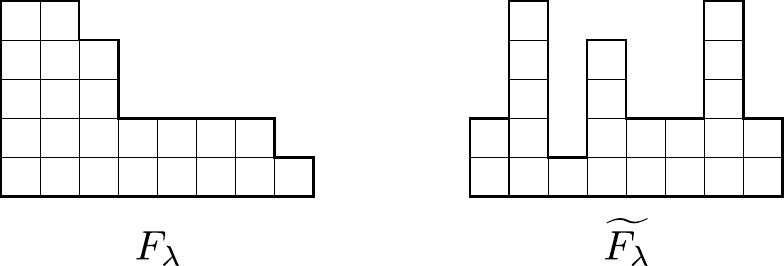}
\caption {Permuting the columns of a Ferrers shape}\label{fig_perm}
\end{figure}
\begin{proof}
First, we summarize some properties of $\widetilde{\FL}$ (which in fact characterize shapes that can be obtained from $\FL$ by permuting columns):
\begin{itemize}
\item $\FL$ and $\widetilde{\FL}$ have the same number of rows, say $k$;
\item the $i^{\text{th}}$ row has the same number of cells in $\widetilde{\FL}$ as in $\FL$, for all $i=1,\dots,k$;
\item the $i^{\text{th}}$ row (from top to bottom) is a subset of the $(i+1)^{\text{th}}$ row in both $\FL$ and $\widetilde{\FL}$
for all $i=1,\dots,k-1$, in the sense that for every cell of the $i^{\text{th}}$ row, there is also a cell in the same column in the $(i+1)^{\text{th}}$ row.
\end{itemize}
We will prove the theorem by induction on the number of rows. (The case $k=1$ is trivial.)
Let $s$ denote the (common) number of cells of the first row in the shapes $\FL$ and $\widetilde{\FL}$.
It is enough to see that for any fixed $0$-$1$ vector $(b_1,b_2,\dots,b_s)$, the number of those $\Gamma$-free $0$-$1$ fillings
of $\widetilde{\FL}$ in which the $j^{\text{th}}$ cell (from left to right) of the first row is filled with $b_j$ (for $j=1,\dots,s$) is the same as
the number of $\Gamma$-free fillings of $\FL$ with the same property.
This is true because such $\Gamma$-free fillings can be obtained for both $\widetilde{\FL}$ and $\FL$ as follows.
Fill the first row with the vector $(b_1,\dots,b_s)$, as required. The $\Gamma$-free property implies that there must be all $0$'s
below every non-rightmost $1$ of the first row (in the same column), so fill these cells with $0$'s.
It is straightforward to check that the remaining cells (i.e.\ those cells in the $2^{\text{nd}},\dots,k^{\text{th}}$ row that have not been filled with $0$)
can be filled with $0$'s and $1$'s arbitrarily with one condition: 
the $0$-$1$ filling of the remaining cells must be $\Gamma$-free (we do not consider the already filled cells here).
Since the two shapes formed by the remaining cells of $\widetilde{\FL}$ and $\FL$ can be obtained from the \emph{same} Ferrers shape (with $k-1$ rows)
by permuting its columns, the induction hypothesis applies, and hence the number of $\Gamma$-free fillings of the remaining cells is the same for both cases.
This concludes the proof.
\end{proof}
This paper just provides a starting point on the subject of enumerating lonesum/$\Gamma$-free fillings,
a number of interesting other shapes can be considered in future research. For example,
a possible next step could be the investigation of skew Ferrers shapes.

\section*{Second author's acknowledgements}

This research was supported by the following grants:
\begin{itemize}
\item Ministry for Innovation and Technology, Hungary grant TUDFO/47138-1/2019-ITM,
\item NKFIH Fund No.\ KH 129597,
\item EU-funded Hungarian grant EFOP-3.6.2-16-2017-00015.
\end{itemize}

\end{document}